\documentclass[12pt]
{article}
\usepackage{amssymb,amsmath}

\usepackage[latin1]{inputenc}
\usepackage[T1]{fontenc}
\usepackage{play}
\usepackage[normalem]{ulem}

\usepackage{graphicx}
\DeclareGraphicsExtensions{.pdf,.png,.jpg}

\graphicspath{ {images/} }

\usepackage{amsmath}
\usepackage{amsthm}
\usepackage{amsfonts}
\usepackage{bbm}
\usepackage{amssymb}
\usepackage{graphicx}
\usepackage{color}

\usepackage{fancybox}
\newlength{\querylen}
\newcommand{\prob}{\mathbb{P}}

\newcommand{\od}{\overset{{\rm d}}{=}}
\newcommand{\dod}{\overset{{\rm d}}{\to}}

\newcommand{\me}{\mathbb{E}}

\newenvironment{myproof}{\noindent {\it Proof} }{\hfill$\Box$}
\newtheorem{thm}{Theorem}
\newtheorem{lemma}[thm]{Lemma}

\theoremstyle{definition}

\theoremstyle{remark}

\begin{document}
\title{
Asymptotics and
Renewal Approximation in the  Online Selection of   Increasing Subsequence
} 
\author{Alexander Gnedin 
and Amirlan Seksenbayev}

\date{\it Queen Mary, University of London}

\maketitle
\begin{abstract}
\noindent
We revisit 
the problem of maximising the expected length of increasing subsequence that can be
selected from { a } marked Poisson process by an online strategy. 
Resorting to a natural size variable, the problem is represented in terms of a controlled { piecewise} deterministic Markov process with decreasing paths. 
Refining known estimates we obtain fairly complete asymptotic expansions
for the moments,
and using a renewal approximation give a novel proof of the 
central limit theorem for the  length of selected subsequence under the optimal strategy.

\end{abstract}

\noindent
\paragraph{1. Introduction}
Suppose a  sequence of  independent  random marks with given continuous distribution  is observed at times of the unit-rate Poisson process. Each time a mark is observed it can be selected or rejected, with every  decision 
becoming immediately final. What is the maximum expected length  $v(t)$ of increasing subsequence, which can be selected over a given horizon $t$ in  {\it online} fashion  by a nonanticipating  strategy?

Samuels and Steele \cite{SS} introduced this stochastic optimisation problem as offspring of its counterpart with   fixed sample size and used similarity between the two problems to obtain the leading asymptotics, $v(t)\sim \sqrt{2t}$ for $t\to\infty$.  
Remarkably,  the square-root   order of growth
was concluded  from superadditivity, similarly to the seminal Hammersley's argument for the longest ({\it offline} in our context) increasing subsequence \cite{Dan}.

The best known to date bounds are
\begin{equation}
\label{BD1}
\sqrt{2t}- \log(1+\sqrt{2t})+\tilde{c}<  v(t)<\sqrt{2t}.
\end{equation}
Analogous upper bounds are common in  folks literature { \cite{ A2, A1, BG, BR, G1}}. 
The lower bound (for $t$ not too small),
with  explicit constant $\tilde{c}$,  appeared in Bruss and Delbaen \cite{BD1}. In the same paper Bruss and Delbaen also obtained estimates on the variance of the length and in \cite{BD2} they
proved  a  functional limit theorem for  this and other characteristics  of the optimal selection process.

The  approach in \cite{BD1, BD2} drew heavily on the concavity of  $v(t)$ and  involved martingale arguments.
In this paper we revisit the problem in light of the  
Samuels-Steele observation \cite{SS} that changing the scale to $z=t^{1/2}$ 
yields an approximate linearisation. 
We improve upon (\ref{BD1})
by obtaining an asymptotic expansion for the optimal expected length
\begin{equation}
\label{AE}
v(t)=
\sqrt{2t}-\frac{1}{12}\log t+c^* { + \frac{\sqrt{2}}{144 \sqrt{t}} +  O(t^{-1})},~~~~t\to\infty,
\end{equation}
similarly refine known estimates of the variance and
show that there is a simple
 selection strategy  within $O(1)$  from the optimum. 
The finer asymptotics (\ref{AE}) up to a bounded term is obtained by bootstrapping the optimality equation similarly to the method used in \cite{BG}.
But justifying convergence of the remainder and the expansion beyond $O(1)$ require much more probabilistic insight.
Our main novelty here is a representation  of the selection problem in terms of  a controlled {piecewise} deterministic Markov process in one dimension. 
 On this way, we will
use comparison with an alternating renewal process
 to explain the logarithmic term in (\ref{AE}) and to give a new proof of the  central limit  theorem for the length of increasing subsequence under the optimal and some suboptimal strategies.

Notation. We will use $\sim$ for asymptotic expansions written without estimate of the remainder, e.g.
 $f(t)\sim f_1(t)+f_2(t)+\dots+f_k(t)$  as $t\to\infty$ means that $ f_{i+1}(t)=o(f_i(t))$ for $1\leq i<k$. We denote $c^*, c_0, c_1$ etc some  absolute constants, and $c$  a constant with context-dependent value.



\paragraph{2. Planar Poisson setup and the leading asymptotics} 
Standardising the distribution of marks to $[0,1]$-uniform leads to a  natural setting of the problem with horizon $t$ in terms of the unit-rate  Poisson random measure in the rectangle $[0,t]\times[0,1]$.
The generic  atom $(s,x)$ is interpreted as mark $x$ observed at time $t$, whereupon a selection/rejection decision must be made solely on { the }base of the allocation of atoms within $[0,s]\times [0,1]$.
A  sequence $(s_1,x_1),\dots,(s_n,x_n)$  of atoms is said to be increasing if it is a chain in the partial order in two dimensions, that is $0<s_1<\dots<s_n$ and $0<x_1<\dots<x_n$. The task 
is to maximise the expected length of an increasing  sequence over  selection strategies adapted to the aforementioned information.

To solve the optimisation problem  it is sufficient to consider a relatively small class of strategies defined recursively 
by means of some {acceptance window} 
 $\psi(t,s,y)$  
satisfying $0\leq\psi(t,s,y)\leq 1-y$ for $0\leq s\leq t<\infty$ and  $y\in[0,1]$.
The corresponding   strategy selects
observation $(s,x)\in[0,t]\times[0,1]$  if and only if $0<x-y\leq \psi(t,s,y)$, where $y$ is the { running maximum}, i.e. the last (hence the highest) mark selected before time $s$, with the convention that $y=0$ if no selections have been made.
Note that the running maximum process and the selected chain uniquely determine one another.

The acceptance window can be regarded as a control function for the 
running maximum, which is a right-continuous Markov process 
$Y=(Y(s), ~0\leq s\leq t)$ starting with $Y(0)=0$, with piecewise constant paths increasing by positive jumps.
At time $s$ in state $y$  a transition  occurs at  rate $\psi(t,s,y)$, and if  $Y$ jumps the increment $Y(s)-Y(s-)$ is uniformly distributed on $[0,\psi(t,s,y)]$. 
For the optimal process, the expected number of jumps is maximal.

Intuitively, a large acceptance window steers $Y$ from $0$ to  about $1$ in just a few jumps. On the other hand,  a small acceptance window makes the jumps rare, so the  
time resource expires before a substantial number of selections is made.

For instance, the {\it greedy} strategy has the largest possible acceptance window  $\psi(t,s,y)=1-y$.   The strategy selects the sequence of records { \cite{Records}}, which has the expected length 
given by the exponential integral function
\begin{equation}\label{greedy}
{\rm Ein}(t)=\int_0^t \frac{1-e^{-s}}{s}\,{\mathrm{d}s}\sim \log t, ~~~t\to\infty.
\end{equation}
The greedy strategy is optimal for $t\leq1.345\cdots$, when the expected number of records is not bigger than 1.

A {\it stationary} strategy  has  acceptance window of the form $\psi(t,s,y)=(1-y)\wedge\delta(t)$, depending neither on the time of observation
nor on the running maximum, as long as $Y$ does not overshoot $1-\delta(t)$. 
 The asymptotic optimality in the principal term is achieved within this class. 
In the remaining part of this section, we sketch a proof which absorbs some ideas from the previous work { \cite{A1, BG, G1, SS}}.

Choose  $\delta=\delta(t)$. 
Up to the first moment when a selected mark exceeds $1-\delta$, 
the running maximum $Y$  coincides  
with a compound Poisson process $S$,  characterised by the jump rate $\delta$ and  the $[0,\delta]$-uniform  distribution of increments.
For  $\delta\to 0$ but $t\delta \to\infty$, the number of  jumps of $S$ over  time $t$ is asymptotic to $t \delta$
(in the mean-square sense), and the number of jumps until  $S$ passes $1-\delta$ is asymptotic to $2/\delta$.
The maximum of $(t\delta)\wedge (2/\delta)$ is attained for $\delta^*(t)= \sqrt{2/t}$, which 
results in the expected length asymptotic to $\sqrt{2t}$. 
After the first selection above $1-\delta^*$  the strategy is greedy, with the expected number of choices being $O(1)$, hence  not affecting the leading asymptotics.

In the sequel under the stationary strategy we shall mean the one with $\delta=\delta^*$, that is with $\psi(t,s,y)=(1-y)\wedge \sqrt{2/t}$.
This strategy maintains a balance between increasing on the marks and time scales, 
so that the running maximum $Y$ fluctuates  about the linear function $s/t$, and
 both resources are exhausted almost simultaneously. 

Let $L_0(t)$ be the length of increasing sequence chosen by the stationary strategy. Representing $L_0(t)$ as a minimum
of two independent renewal processes we have a limit 
\begin{equation}\label{limsta}
\sqrt{3}\,\frac{L_0(t)-\sqrt{2t}}{(2t)^{1/4}}\dod \eta,
\end{equation}
with $\eta= \xi_1\wedge (\sqrt{3}\xi_2)$, where  $\xi_1$ and $\xi_2$ are independent standard normal variables.
Specialising formulas for moments found in  \cite{Nada},
${\mathbb E}\eta= -\sqrt{2/\pi}$, ${\rm Var}(\eta)=2-2/\pi$.

Now, for the compound Poisson process $S$ controlled 
by $\delta^*$,  the expected number of jumps is exactly  equal to $\sqrt{2t}$.
To prove the upper bound (\ref{BD1}) we will show that  $S$ can be identified with the optimal chain 
in an online selection problem with a weaker {mean-value} constraint.
To that end, consider a problem of  online selection  from the Poisson random measure in  unbounded domain $[0,t]\times [0,\infty)$,  but with the restriction that
 observation $(s,x)$ is available for selection only if $0<x-y\leq 1$, where $y$ is the running maximum at time $s$.  Suppose the objective is to 
maximise the expected length of the selected chain subject to the constraint that the {\it mean} mark of the ultimate selection does not exceed  $1$.
Clearly, every strategy with choices from the bounded rectangle $[0,t]\times[0,1]$ is admissible also in the extended scenario; in particular, it has the last 
selected mark not exceeding 1 almost surely. In the extended setting, the observations satisfying the chain condition arrive by a unit-rate Poisson process independent of the marks.
Whichever selection strategy, if $y$ is the last mark chosen before time $s$, the next (if any) mark which comes in question is $y+\xi$, where $\xi$ is $[0,1]$-uniformly distributed  and independent of the previous observations.
Let $(s_i, \xi_i)$ be the increasing sequence of observation times of such marks and $\xi_i$ their associated  uniform variables. A decision on observation at time $s_i$ can be encoded into  a 0-1 random variable $\pi_i$
adapted to the Poisson random measure within $[0,s_i]\times[0,\infty)$.
 With this notation  the  task becomes a  version of the  knapsack problem 
$${\mathbb E}\left(\sum_i \pi_i \right)\to\max,~~~{\rm subject~to~~}      {\mathbb E}\left(\sum_i \pi_i \xi_i      \right)\leq 1,~~~\pi_i\in\{0,1\},$$
 with the Lagrange function 
$${\mathbb E}\left(\sum_i \pi_i (1- \lambda\xi_i)\right).$$
Here, $\pi_i$'s enter linearly hence  the maximum
is attained by the indicators $\pi_i^*=1(\xi_i<1/\lambda)$. The multiplier $\lambda^*=(t/2)^{1/2}$ is found from the constraint 
$$1={\mathbb E}\left(\sum_i \xi_i 1(\xi_i<1/\lambda)\right)=\frac{t}{2\lambda^2}.$$ 
Since $1/\lambda^*=\delta^*$, the selected chain has the same distribution as the compound Poisson process $S$, whence the upper bound in (\ref{BD1}).

\paragraph{3. The optimality equation} The stationary strategy lacks the following self-similarity feature inherent to the overall optimal strategy.
If at time $s$ the running maximum is $y$,  further selections are to be made from $[s,t]\times [y,1]$, which is  an  independent subproblem, equivalent 
to the original problem in $[0,t']\times[0,1]$ with $t'=(t-s)(1-y)$. 
This implies that it is sufficient to optimise over the class of {\it self-similar} strategies with acceptance window of the form
\begin{equation}\label{psi}
\psi(t,s,y)=(1-y) \ \varphi((t-s)(1-y))
\end{equation}
for some $\varphi:[0,\infty)\to[0,1]$. Such a strategy accepts mark $x$ at time $s$ if and only if 
$$0<\frac{x-y}{1-y}\leq \varphi((t-s)(1-y)).$$

The  dynamic  programming principle leads to the optimality equation 
for the maximal expected length $v(t)$,
\begin{equation}\label{uDP0}
v'(t)=\int_0^1 (v(t(1-x))+1-v(t))_+ \,{ \mathrm{d}x}, ~~~v(0)=0,
\end{equation}
The optimal acceptance window is
$\psi^*(t,s,y)=(1-y)\varphi^*((t-s)(1-y))$, where 
$\varphi^*(t)=1$ 
if $v(t)\leq1$ (when $v(t)$ is given by (\ref{greedy})),
and otherwise   $\varphi^*(t)$ is defined implicitly as  
the unique solution to 
\begin{equation*} 
v(t(1-x))+1-v(t)=0.
\end{equation*}

See { \cite{BD1}} for the derivation of (\ref{uDP0}), properties and estimates of $\varphi^*(t)$ and  $v(t)$.
Our focus is on the  asymptotic expansion for large $t$.

With the change of variable
$$u(z):=v(z^2)$$
the optimality equation  (\ref{uDP0}) becomes equation of convolution type 
\begin{equation}\label{uDP1}
u'(z)=4\int_0^z (u(z-y)+1-u(z))_+(1-y/z) \ { \mathrm{d}y}, ~~~u(0)=0.
\end{equation}
We set $\theta^*(z)=z$ if  $u(z)\leq 1$, and otherwise define $\theta^*(z)$ to be the unique solution to
\begin{equation}\label{thetaeq}
u(z-y)+1-u(z)=0.
\end{equation}
By monotonicity we can re-write (\ref{uDP1}) as
\begin{equation}\label{uDP2}
u'(z)=4\int_0^{\theta^*(z)} (u(z-y)+1-u(z))(1-y/z) \ { \mathrm{d}y}, ~~~u(0)=0.
\end{equation}

\vskip0.5cm
\noindent
{\bf 4. Asymptotics I}
Let ${\cal I}g$ be the integral operator  acting on functions $g\in C^1[0,\infty)$ as  
$${\cal I}g(z)=4\int_0^z (g(z-y)+1-g(z))_+(1-y/z)\ { \mathrm{d}y}.$$ 
In this notation  equation (\ref{uDP1}) becomes $u'={\cal I}u$.
Note that we can write 
 ${\cal I}g(z)=4\int_0^{\theta(z)}(\cdots){ \mathrm{d}y}$ 
 if the equation  $g(z-y)+1-g(z)=0$ has a unique solution $y=\theta(z)$.

The following lemma resembles a familiar comparison method to  estimate solutions of differential equations (see \cite{DeB}, section 9.1).  

\begin{lemma}\label{L1} If $g'(z)>{\cal I}g(z)$ for all sufficiently large $z$ then $\limsup\limits_{z\to\infty} (u(z)-g(z))<\infty$.
Likewise, if $g'(z)<{\cal I}g(z)$ for all sufficiently large $z$ then $\liminf\limits_{z\to\infty} (u(z)-g(z))>-\infty$.
\end{lemma}
\begin{myproof}. Observe that ${\cal I}g={\cal I}(g+c)$ for constant $c$. Since $u(z)$ increases to $\infty$ we may choose $z_0$ large enough to satisfy both $z-\theta^*(z)>z_0$ and $g'(z)>{\cal I}g(z)$ for $z>2z_0$.
Assume to the contrary that $\limsup\limits_{z\to\infty} (u(z)-g(z))=\infty$. Choosing $c$ large enough we will achieve that  
$u(z)<g(z)+c$ for $z_0<z<z_1$,  and $z_1>2z_0$  for  $z_1:=\min\{z>z_0:  u(z)=g(z)+c\}$, which exists by the assumption.  Then for $y<z_1$
$$(u(z_1-y)+1-u(z_1))_+<((g(z_1-y)+c)+1-(g(z_1)+c))_+,~~~$$
hence  my monotonicity of the integral $u'(z_1)={\cal I}u(z_1)\leq {\cal I}(g+c)(z_1)={\cal I}g(z_1)<g'(z_1)$.
But this is a contradiction since 
$u'(z_1)\geq (g+c)'(z_1)= g'(z_1)$ by definition of $z_1$ as the   location where $u$ first reaches $g+c$.
The second part of the lemma is argued similarly.
\end{myproof}
\vskip0.2cm

We will now compare the solution to (\ref{uDP1})  with various test functions.  
Let $u_1(z):=\alpha_1z$. We have $u'_1(z)=\alpha_1$ and for 
$\theta_1(z):=1/\alpha_1$, 
$${\cal I}u_1(z)=4\int_0^{\theta_1(z)} (u_1(z-y)+1-u_1(z))(1-y/z) \ { \mathrm{d}y}\to  \frac{2}{\alpha_1}, ~~~~z\to\infty.$$   
The match  $\alpha_1= {2}/{\alpha_1}$ occurs at $\alpha_1^*:=\sqrt{2}$, thus
by the lemma $\limsup\limits_{z\to\infty}( u(z)-u_1(z))<\infty$ for $\alpha_1>\sqrt{2}$ and therefore
$\limsup\limits_{z\to\infty} u(z)/z\leq \sqrt{2}$.
Likewise,
the second part of { the} lemma yields $\liminf\limits_{z\to\infty} u(z)/z\geq \sqrt{2}$. These bounds imply 
$u(z)\sim \sqrt{2}\,z$.

We try next  functions $u_2(z):=\sqrt{2}z+\alpha_2\log (z+1)$ (we take $\log(z+1)$ and not $\log z$  to avoid the annoying singularity at $0$). 
Solving $u_2(z-y)+1-u_2(z)=0$, for large $z$ we get expansion
\begin{equation}\label{theta2}
\theta_2(z)\sim \frac{1}{\sqrt{2}}-\frac{\alpha_2}{2(z+1)}.
\end{equation}
We may proceed with only the first term in (\ref{theta2}) since the second makes a negligible $O(z^{-2})$ contribution to ${\cal I}u_2(z)$ which expands as
\begin{equation*}
{\cal I} u_2(z)\sim \sqrt{2}-\left( \frac{1}{3}+\alpha_2\right)\frac{1}{z+1}.
\end{equation*}
With $u_2'(z)=\sqrt{2}+\alpha_2/(z+1)$ the match occurs when 
$$\alpha_2=-\left( \frac{1}{3}+\alpha_2\right),$$
that is for $\alpha_2^*:=-1/6$. It follows from the lemma that $(u(z)-\sqrt{2}z)/\log(z+1)\to \alpha_2^*$, that is 
$$u(z)\sim \sqrt{2}z-\frac{1}{6}\log z.$$

To further refine the approximation we try  
\begin{equation}\label{u3}
u_3(z):=\sqrt{2}\,z-\frac{1}{6} \log (z+1)+\frac{\alpha_3}{z+1}.
\end{equation}
This time we need to calculate with higher precision, hence  take two terms 
\begin{equation}\label{theta3}
\theta_3(z)\sim \frac{1}{\sqrt{2}} +\frac{1}{12(z+1)}.
\end{equation}
Expanding the integrand and integrating:
$${\cal I}u_3(z)\sim \sqrt{2}-\frac{1}{6(z+1)}+\left(\alpha_3-\frac{1}{36 \sqrt{2}}- \frac{1}{3}\right)\frac{1}{(z+1)^2}\,.$$
To match with 
$$u_3'(z)= \sqrt{2}-\frac{1}{6(z+1)}  -\frac{\alpha_3}{ (z+1)^{2}}$$
we must choose   $\alpha_3^*:= 1/6 +\sqrt{2}/144$.
Taking $\alpha_3$ bigger or smaller than $\alpha_3^*$, allows us to sandwich $u$.
However, our 
 comparison method based on the lemma only yields
\begin{equation}\label{u3appr}
u(z)= \sqrt{2}\,z-\frac{1}{6}\log z+O(1), ~~~~z\to\infty,
\end{equation}
 since the third  term in (\ref{u3}) is  already bounded. 
A different approach will be applied to show  convergence of the $O(1)$ remainder.

\vskip0.5cm
\noindent
{\bf 5. Piecewise deterministic Markov process}
We represent next the selection problem by means of a piecewise deterministic Markov process in one dimension.

Let $\theta: (0,\infty)\to (0,\infty)$  be a function  satisfying  $0<\theta(z)\leq z$, and let
\begin{equation*}
\lambda(z):=\theta(z)-\frac{\theta^2(z)}{2z}.
\end{equation*}

The following rules define 
a { piecewise} deterministic Markov process $Z$  on $[0,\infty)$ with continuous drift component and random instantaneous jumps:  
\begin{itemize} 
\item[(i)] the process decreases continuously with unit speed,
\item[(ii)] the jumps are negative and  occur at rate $4\lambda(z)$, for $z>0$,
\item [(iii)] if a jump from state $z$ occurs, the jump size has density $(1-y/z)/\lambda(z)$ with support $[0,\theta(z)]$,
\item [(iv)] the process terminates upon reaching $0$.
\end{itemize}

We denote $Z|z$ this process starting in position $z$. The path of $Z|z$
can be constructed by thinning the set of arrivals of 
an inhomogeneous marked Poisson process with intensity  (ii) and marks distributed as in (iii). 
The following procedure is similar to many familiar parking, packing and scheduling models in applied probability.
Let $z_1$ be the rightmost arrival on $[0,z]$ with some mark $y_1$. 
Call $(z_1,z]$ a {\it drift interval} comprised of {\it drift points}, and call
 $z_1$ a {\it jump point}.
Remove all arrivals from the {\it gap} $(z_1-y_1,z_1]$. Then iterate thinning arrivals of the Poisson process, to the left of $z_1-y_1$ in place of $z$.
If after  some iteration  a jump point $z_k$ cannot be found, the drift interval extends from the last defined point $z_{k-1}-y_{k-1}$ (or $z$ in case $k=1$)
to $0$.  
The union of drift intervals corresponds to the range of  $Z|z$, while the gaps are skipped by jumps. 
The jump points  of $Z|z$ divide $[0,z]$ in {\it  cycles}. A cycle, in the right-to-left order, is comprised of  a drift interval followed by a gap. The exception is the leftmost drift interval adjacent  to $0$.
In this picture, the time variable is unambiguously introduced  by requiring that the time to pass $z'\leq z$ is equal 
to the Lebesgue measure of the range intersected with  $[z',z]$.

To  connect to the increasing subsequence problem choose horizon $t$ and  
let $Y$ be the running maximum process under some self-similar strategy (\ref{psi}).  Note that $(1-Y(s))(t-s)$ is the area of  the rectangle, from which selections after time $s$ can be made. Let 
$$\widetilde{Z}(s):=\sqrt{(1-Y(s))(t-s)}, ~~~s\in[0,t],$$
which is a drift-jump process decreasing from $t^{1/2}$ to $0$, with negative jumps $\Delta \widetilde{Z}(s)= \widetilde{Z}(s)-\widetilde{Z}(s-)$ at times of selection.  
The decay of $\widetilde{Z}$ due to the drift is  a strictly increasing continuous process
$$\sigma(s):=t^{1/2}-\widetilde{Z}(s)+\sum_{s'\leq s}\Delta\widetilde{Z}(s').$$
For $\sigma^{\leftarrow}$ the inverse function to $\sigma$, consider  the time-changed process
\begin{equation}\label{Z}
Z(q)=\widetilde{Z}(\sigma^{\leftarrow}(q)),~~~q\leq\sigma(t).
\end{equation}
Identifying the drift rate and jump distribution it is seen that (\ref{Z}) is the process $Z|\sqrt{t}$,
with $\theta$ found by matching the jump rates as 
$$
4\lambda(z)=2z \varphi(z^2).
$$
In particular, $Y$ over horizon $t=z^2$ has the same number of jumps as $Z|z$.
This reduces the optimal selection problem with horizon $t$ to  choosing a control function $\theta$ with the objective to maximise the expected number of jumps of $Z|\sqrt{t}$.

Denote $N_\theta(z)$ the number of jumps of the process  $Z|z$ steered by  given function $\theta$, and let 
$u_\theta(z):={\mathbb E}N_\theta(z)$. With probability $4\lambda(z) { \mathrm{d}z}$ the process moves from a small vicinity of $z$ to $z-y$, with $y$ sampled from the density in (iii), in which case
the expected number of  jumps is equal to $u_\theta(z-y)+1$. Otherwise, the process drifts through to $z- { \mathrm{d}z}$. This decomposition readily yields equation
\begin{eqnarray}\label{w}
u_\theta'(z)=
4 \int_0^{\theta(z)} (u_\theta(z-y)+1-u_\theta(z)){(1-y/z)} \ { \mathrm{d}y},
 ~~~u_\theta(0)=0.
\end{eqnarray}

For the general $\theta$,  the integrand in (\ref{w}) need not be positive and even monotonicity of $u_\theta$ may not hold.
In purely analytic terms, for any fixed $z$, maximising $u_\theta(z)$ over admissible $\theta:[0,z]\to (0,\infty)$ is the problem of calculus of variations. 
The  solution is  $\theta=\theta^*$, defined implicitly by       
 the optimality equation (\ref{uDP1}) and (\ref{thetaeq}).

We shall assume throughout that $\theta$  is bounded and differentiable.
That the optimal $\theta^*$ is bounded can be seen at this stage of our analysis  from  (\ref{thetaeq}) and (\ref{u3appr}).

The asymptotic comparison method based on an analogue of Lemma \ref{L1}
works for (\ref{w}) as well. In particular, for 
\begin{equation*}
\theta_0:=\sqrt{1/2}\wedge z,
\end{equation*} 
we obtain the same expansion as (\ref{u3appr}).



 The decreasing sequence of jump points of $Z|z_0$ is a Markov chain with terminal state $0$.
 Let $U_\theta(z_0,\cdot)$ be the occupation measure on $[0,z_0]$ counting the expected number 
of jump points, in particular $U_\theta(z,[0,z])=u_\theta(z)$. Denote $p(z_0,z),$ for $0\leq z\leq z_0$ the probability that $z$ is a drift point, in particular $p(z_0,z_0)=p(z_0,0)=1$. There is a jump point within $ { \mathrm{d}z}$ only if $z$ does not belong to a gap, hence
the occupation measure has a density which factorises as
\begin{equation*}
U_\theta(z_0,dz)=4\lambda(z)p(z_0,z) \ {\mathrm{d}z}, ~~~0\leq z\leq z_0.
\end{equation*}

\begin{lemma}\label{cover}
 There exists a pointwise limit $p(z):=\lim\limits_{z_0\to\infty} p(z_0,z)$, which satisfies
$$
|p(z_0,z)-p(z)|< ae^{-\alpha(z_0-z)}, ~~~~~0<z<z_0,
$$
with some positive constants $a$ and $\alpha$.
\end{lemma}
\begin{proof}
The proof is by coupling.  Choose constant $\overline{\theta}$ big enough to have  $\sup\theta(z)<\overline{\theta}$.
Fix $z<z_0<z_1$ with $z>2\overline{\theta}$ (the latter assumption does not affect the result).
Consider two independent processes  $Z_0$ and $Z_1$ with  $Z_0\stackrel{d}{=}Z|z_0, ~Z_1\stackrel{d}{=}Z|z_1$.
Define $Z'$ by running the process $Z_1$ until it hits a drift point $\xi$ of $Z_0$ , then from this point on switch over to running 
$Z_0$. 
Such a point $\xi$ exists since both processes have a gap adjacent to $0$.
By the strong Markov property, $Z'$ has the same distribution as $Z_1$. If the coupling occurs at some $\xi\in [z,z_0]$, the point $z$ is of the same type (drift or jump) for
both $Z'$ and $Z_0$.

The coupling does not occur within $[z,z_0]$ only if $Z_0$ and $Z_1$ have no common drift points within these bounds.
Given that $y>z$ is a drift point, the probability that the drift interval covering $y$ extends to the left over $y-\overline{\theta}$ is at least $\pi$, for some constant $\pi>0$. 
This follows since the length of drift interval dominates stochastically an exponential random variable with rate 
$\sup 4\lambda(z)<\infty$. 
In particular, the rightmost drift interval, adjacent to $z_0$, is shorter than $\overline{\theta}$ with  probability at most  $1-\pi$, in which case the rightmost cycle is shorter  than $2\overline{\theta}$.
Given $\xi$ is not in the first cycle, the probability that $\xi$ is not in the second is again at most $1-\pi$, in which case also the second cycle   is shorter  than $2\overline{\theta}$.
Continuing so forth we see that $\xi\notin[z,z_0]$ with probability at most $(1-\pi)^k$ for $k=\lfloor (z_0-z)/(2\overline{\theta})\rfloor$. 
This readily implies an exponential bound $|p(z_0, z)-p(z_1,z)|<ae^{-\alpha(z_0-z)}$, uniformly in $z_1>z_0$.
Sending $z_0\to\infty$ we see that $p(z_0,z)$ is a Cauchy sequence, whence the claim.
\end{proof}

In the terminology of random sets, $p(z_0,\cdot)$ is the coverage function (see \cite{Molchanov} p. 23) for the range of $Z|z_0$. 
As $z_0\to\infty$ the range converges weakly to a random set ${\cal Z}\subset [0,\infty)$, comprised of  infinitely many intervals separated by gaps. 
Indeed, let $A(z_0,z)\leq z$ be the maximal point of the range of $Z|z_0$ within $[0,z]$, for $z\leq z_0$. The coupling argument in the lemma also 
shows that $A(z_0,z)$ has a weak limit, $A(z)$, which is sufficient to justify convergence of the range intersected with $[0,z]$, due to the Markov property.
By Sheff{\'e}'s lemma $U_\theta(z_0,\cdot)$ converges weakly to some $U_\theta$, which is the occupation measure for the point process of left endpoints of intervals making up $\cal Z$.

\vskip0.5cm
\noindent
{\bf 6. Reward processes}  Suppose each jump point of $Z|z$ is weighted by some location-dependent reward $r$. 
Let $w_{\theta,r}(z)$ be the total expected reward accumulated by $Z|z$ controlled  by $\theta$. By analogy with (\ref{w}) we have  equation
\begin{equation}\label{wr}
w_{\theta,r}'(z)=4 \int_0^{\theta(z)} (w_{\theta,r}(z-y)+r(z)-w_{\theta,r}(z)){(1-y/z)}{} \ { \mathrm{d}y},\\
 ~~~w(0)=0.
\end{equation}
 On the other hand, we can write $w_{\theta,r}(z)$ as the average over the occupation measure,

\begin{equation}\label{FC}
w_{\theta,r}(z)=\int_0^z r(y)U_\theta(z,{\mathrm{d}y})=4\int_0^z r(y)\lambda(y) p_\theta(z,y) \ { \mathrm{d}y}.
\end{equation}

\begin{lemma}\label{convrem}
For $r$ integrable function, the solution to {\rm (\ref{wr})} has a finite limit
\begin{equation}\label{wU}
\rho_{\theta,r}:=\lim_{z\to\infty} w_{\theta,r}(z)=
{ 4}\int_0^\infty r(y)\lambda(y)p_\theta(y) \ { \mathrm{d}y}.
\end{equation}
If $|r(z)|=O(z^{-\beta})$ as $z\to\infty$ for some $\beta>1$ then $|w_{\theta,r}(z)-\rho_{\theta,r}|=O(z^{-\beta+1})$.

\end{lemma}
\begin{proof}
Since $p(z_0,z)\lambda(z)<\overline{\theta}$
the existence of limit follows  from  (\ref{FC}), (\ref{wU}) and Lemma \ref{cover} by the dominated convergence. 
The convergence rate is estimated  by splitting the difference as
$$
\rho_{\theta,r}-w_{\theta,r}(z)={ 4} \int_0^{z/2} r(y)\lambda(y)(p_\theta(z,y)-p_\theta(y))\ { \mathrm{d}y}+{ 4}\int_{z/2}^\infty r(y)\lambda(y)p_\theta(y)\ { \mathrm{d}y},
$$
where 
 the second integral is of the order $O(z^{-\beta+1})$ while the first is of the lesser order $O(e^{-\alpha z/2})$
by Lemma \ref{cover}. 
\end{proof}


\vskip0.5cm
\noindent
{\bf 7. Asymptotics II} Differentiating  (\ref{uDP1}) with an account of  (\ref{thetaeq}) we obtain

\begin{equation}\label{u''}
u^{''}(z)=4\int_0^{\theta^*} (u'(z-y)+r(z)-u'(z))(1-y/z)\ { \mathrm{d}y}, ~~~u'(0)=0.
\end{equation}
where
$$
r(z)=\frac{4}{\lambda(z)z^2}\int_0^{\theta^*(z)} (u(z-y)+1-u(z))y \ { \mathrm{d}y}.
$$ 
Since $\theta^*(z)=z$ for small $z$
this  has a simple pole  at 0,  but the singularity is compensated in (\ref{FC}), so Lemma \ref{convrem} and (\ref{u3appr}) ensure that 
\begin{equation}\label{expder}
u'(z)=\sqrt{2}+O(z^{-1}).
\end{equation}

With (\ref{expder}) at hand, 
expanding in (\ref{thetaeq}) we get $\theta^* (z)=\sqrt{1/2}+O(z^{-1})$. Replacing $\theta^*$ by $\sqrt{1/2}$ in (\ref{uDP1}) incurs remainder of smaller order $O(z^{-2})$ because $\theta^*$ is the stationary point of the integral.
Recalling that $u_2$ (with $\alpha_2^*=-1/6$) satisfies $u_2'(z)={\cal I}u(z)+O(z^{-2})$, for the difference $w=u-u_2$ we obtain equation (\ref{wr}) with $r(z)=O(z^{-2})$, hence $u(z)-u_2(z)$ by Lemma \ref{convrem} 
approaches a finite limit at rate $O(z^{-1})$ as $z\to\infty$. This proves an  expansion 
\begin{equation}\label{expan2}
u(z)= \sqrt{2}\,z-\frac{1}{6}\log z+c^*+O(z^{-1}), ~~~z\to\infty
\end{equation}
with some constant $c^*$.

Our methods are not geared to identify $c^*$, because the initial value $u(0)=0$ was nowhere used, but changing it (e.g., by resorting to a selection problem with terminal payoff) will result in adding $u(0)$ to $c^*$.
Nevertheless, with some more effort it is possible to  go beyond $O(1)$. Let us first estimate the variation of $u'$.

\begin{lemma}\label{variation}

For fixed $\overline{d}>0$, as $z\to\infty$
$$\sup_{0\leq d\leq  \overline{d}}|u'(z+d)-u'(z)|=O(z^{-2}).$$
\end{lemma}
\begin{proof}
Using the integral representation (\ref{FC}) of $u'$ with $r(z)=O(z^{-2})$, write
\begin{eqnarray*}
u'(z+d)-u'(z)=\int_z^{z+d} r(y)p(z+d,y)\lambda(y)\ { \mathrm{d}y}+ \int_0^{z} |p(z+d,y)-p(z,y)|\lambda(y)r(y)\ { \mathrm{d}y}.
\end{eqnarray*}
The first integral is obviously $O(z^{-2})$ uniformly in $d\leq\overline{d}$.
By Lemma \ref{cover} the second is estimated as
$$c\int_0^z e^{-\alpha(z-y)} {(y^2+1)^{-1}} { \mathrm{d}y}=O(z^{-2})$$
using Laplace's method.
\end{proof}
The lemma applied to the right-hand side of (\ref{u''}) gives $u''(z)=O(z^{-2})$. In (\ref{uDP2}) we replace $\theta^*$ by $\sqrt{1/2}$, expand $u(z-y)-u(z)=-yu'(z)+O(z^{-2})$ and integrate to obtain with some algebra
$$u'(z)=\sqrt{2}-\frac{1}{6z}+O(z^{-2}).$$ 
Expanding similarly in (\ref{thetaeq}) we get a finer formula for the optimal control function
\begin{equation}\label{theta-2}
\theta^*(z)=\frac{1}{\sqrt{2}}+\frac{1}{12z}+O(z^{-2}), ~~~z\to\infty,
\end{equation}
in  accord with (\ref{theta3}). Since $u_3'(z)={\cal I}u_3(z)+O(z^{-3})$ the difference $w=u-u_3$ satisfies (\ref{wr}) with $r(z)=O(z^{-3}))$, hence invoking Lemma \ref{convrem} we obtain $u(z)-u_3(z)=\hat{c}+O(z^{-2})$
for some constant $\hat{c}$.
This must agree with (\ref{expan2}), therefore $\hat{c}=c^*$.  Thus we have shown
 
\begin{thm}\label{opt-subopt} For the optimal process, the control function  $\theta^*$  satisfies
{\rm (\ref{theta-2})}, and the expected number of jumps has expansion
\begin{equation}\label{expan3}
u(z)=\sqrt{2}\,z-\frac{1}{6}\log z +c^*+\frac{\sqrt{2}}{144\,z}+O(z^{-2}),~~~z\to\infty.
\end{equation}
\end{thm}

To appreciate the effect of the second term in (\ref{theta-2}) it is helpful to consider
control functions of the kind
\begin{equation}\label{thetagamma}
\theta(z)\sim \frac{1}{\sqrt{2}}+\frac{\gamma}{z}\,,~~~z\to\infty.
\end{equation}
The parameter appears in the asymptotics of solutions to (\ref{w})  as  
\begin{equation}\label{gaga}
u_\theta(z)\sim\sqrt{2}\,z-\frac{1}{6}\log z +c+    \left(\frac{\sqrt{2}}{72}-\frac{\sqrt{2}\gamma}{6} +\sqrt{2}\gamma^2 \right)\,\frac{1}{z}\,
\end{equation}
(whichever $u_\theta(0)$ that only affects the constant).

Constant $c$ in (\ref{gaga}) does not exceed $c^*$ in (\ref{expan3}), but  the relation between the $z^{-1}$-terms can be the opposite.
For instance, for $\theta_0(z)=\sqrt{1/2}\wedge z$ we have
\begin{equation*}
u_{\theta_0}(z)=\sqrt{2}\,z-\frac{1}{6}\log z +    c_0 +  \frac{\sqrt{2}}{72\,z}               +O(z^{-2}).
\end{equation*}

\vskip0.5cm
\noindent
{\bf 8. The variance} For $N_\theta(z)$, the number of jumps of $Z|z$ driven by $\theta$,
 let  $w(z)={\mathbb E}(N_\theta(z))^2$ be the second moment. 
This function satisfies
\begin{equation*}
w'(z)=4\int_0^\theta [w(z-y)+(1+2u_\theta(z-y))-w(z)](1-y/z)\ { \mathrm{d}y},~~~w(0)=0.
\end{equation*}
Integrating the inhomogeneous term this can be reduced to the form (\ref{wr}), but with $r(z)$ of the order of $z$. This equation has properties similar to (\ref{uDP1}), that is
adding a constant to $w$ also gives a solution with some other initial value,
and the right-hand size increases in $w(z-y)$. Hence an analogue of Lemma \ref{L1} can be applied to compare $w$ with various test functions.

We shall consider first the case of optimal $\theta=\theta^*$.
It is an easy exercise to see that $w(z)\sim 2z^2$, hence the leading term in the integrand
is $-4zy+2\sqrt{2}z$, which vanishes at $y=\sqrt{1/2}$. For this reason  the $O(z^{-2})$ remainder  in (\ref{theta-2}) 
will contribute to the solution only $O(1)$, and not $O(\log z)$ as one might expect. 
Using  this fact and (\ref{theta-2})  it is possible to match  the sides of the equation 
by  selecting   coefficients of the
 test function
 $$\hat{w}(z)={ 2}z^2+a_1 z\log z+ a_2 z+a_3(\log z)^2 +  a_4\log z,$$
achieving that  the difference
$w(z)-\hat{w}(z)$ satisfies  an equation of the type (\ref{wr})  with
$r(z)=O(z^{-2}\log z)$. 
Then applying   Lemma \ref{convrem}, $w(z)-\hat{w}(z)\sim c_1+z^{-1}\log z$.
With some help of 
{\tt Mathematica} 
 we arrived at
\begin{equation}\nonumber
w(z)\sim 2z^2-
\frac{\sqrt{2}}{3} z\log z+\left(\frac{\sqrt{2}}{3}+2\sqrt{2}c^*    \right)z+\frac{1}{36} (\log z)^2+\left(\frac{1}{36}-\frac{c^*}{3}\right)\log z + c_1.
 \end{equation}
From this and    (\ref{expan2}) for ${\rm Var}(N_{\theta^*}(z))= w(z)-u^2(z)$ we obtain
$$
{\rm Var}(N_{\theta^*}(z))=  \frac{\sqrt{2}z}{3}+\frac{1}{36}\log z+c_2+O\left( z^{-1}\log z   \right)      \,,~~~z\to\infty.
$$
with 
$c_2:=c_1 -(c^*)^2-{1}/{36}$.
In fact, the value of $c^*$ in (\ref{expan2}) impacts $c_1$ but not  $c_2$, because the latter is invariant 
 under  shifting $u(0)$.

For the general control functions,  the variance is very sensitive to the behaviour of $\theta$. The convergence  $\theta(z)\to\sqrt{1/2}$ alone does not even ensure  that  
$O(z)$ is the right {\it order}  for
${\rm Var}(N_{\theta^*}(z))$. 
If (\ref{thetagamma}) holds we have the asymptotics 
$${\rm Var}(N_\theta(z))\sim \frac{\sqrt{2}z}{3}+\left( \frac{1-8\gamma}{12}\right)\log z\,,~~~z\to\infty.$$

\vskip0.5cm
\noindent
{\bf 9. CLT for the number of jumps} In this section we denote $N(z)$ the number of jumps of $Z|z$ with  some control function
satisfying 
\begin{equation}\label{LT}
\theta(z)=\frac{1}{\sqrt{2}}+O\left( {z^{-1}}\right), ~~{\rm hence}~~\lambda(z)=\frac{1}{\sqrt{2}}+O( {z^{-1}}), ~~~z\to\infty.
\end{equation}

Denote $J_z$  the size of the generic gap having the right endpoint $z$,   with density 
\begin{equation*}
\prob(J_z\in { \mathrm{d}y})=\frac{1-y/z}{\lambda(z)}, ~~~0\leq y\leq \theta(z),
\end{equation*}
and let $D_z$ be the size of the generic drift interval with survival function
\begin{equation}\label{distD}
\prob(D_z\geq y)=\exp\left( -\int_{z-y}^z 4\lambda(s) \ { \mathrm{d}s} \right), ~~~0\leq y\leq z.
\end{equation}
The size of the generic cycle with the right endpoint $z$ can be written as
$$D_z+J_{z-D_z},$$
where $D_z$ and the family of variables  $J_{\cdot}$ are independent, and we set $J_0=0$.

For large $z$,
the expected values of $J_z$ and $D_z$ are about equal, suggesting that about a half of $[0,z]$ is covered by drift and another half is skipped by jumps.
This resembles the behaviour of the stationary (and, as seen from  \cite{BD2}, also of the optimal) selection process in the planar Poisson setting, where the balance is  kept on two scales.

It is useful to see how the mean sizes of gaps and drift intervals 
depend on  $\theta=\theta(z)$:
$$\me J_z= \frac{\theta}{2}-\frac{\theta^2}{12z}+O(z^{-2}),~~~\me D_z =\frac{1}{4\theta}+\frac{1}{8z}+O(z^{-2}).$$
The leading terms match for $\theta$ as in (\ref{LT}), in which case the mean size of a cycle is 
$$\me (J_z+D_z)\sim \frac{1}{\sqrt 2}+\frac{1}{12 z}+O(z^{-2}),$$ 
regardless of the $O(z^{-1})$ term in (\ref{LT}). This expansion explains why  
the second term  in (\ref{expan3}) is $O(\log z)$ (but falls short of explaining the coefficient $-1/6$),
and why the suboptimal strategy in Theorem \ref{opt-subopt} is $O(1)$ from the optimum.


From the convergence of parameters (\ref{LT}) it is clear that 
as $z\to\infty$
\begin{equation*}
D_z\dod \frac{E}{2\sqrt{2}}\,, ~~~J_z\dod \frac{U}{\sqrt{2}},
\end{equation*}
and, observing the joint convergence of $(D_z, J_{z-\,\cdot})$, also that
\begin{equation}\label{cycleconv}
D_z+J_{z-D_z}\dod \frac{E}{2\sqrt{2}}  +\frac{U}{\sqrt{2}}\,,
\end{equation}
where $U\od {\rm Uniform}[0,1]$ and $E\od{\rm Exponential}(1)$ are independent.

The weak convergence (\ref{cycleconv}) of cycle sizes suggests that the behaviour of $N(z)$ for large $z$ can be deduced from that 
of   a renewal process  with the generic step 
$$H:=\frac{E}{2\sqrt{2}}  +\frac{U}{\sqrt{2}}\,$$
which has moments
$$\mu:=\me \,H =\frac{1}{\sqrt{2}}, ~~\sigma^2:={\rm Var}\left(H\right)=\frac{1}{6},~~~\frac{\sigma^2}{\mu^3}=\frac{\sqrt{2}}{3}.$$

Specifically, for 
 the renewal process $R(z):=\max\{n: H_1+\dots+H_n\leq z\}$, with $H_j$'s being i.i.d. replicas of $H$,
we have the familiar CLT
$$\frac{R(z)-z\mu^{-1}}{\sigma\mu^{-3/2}\sqrt{z}}   
\dod    {\cal N}(0,1),$$
and one can expect that the same limit holds for $N(z)$.
This line should be pursued with care,
because local discrepancies may accumulate on the large scale and bias {centring} or even the type of the limit distribution.
In the approach taken in the sequel,  we amend some details of the method of stochastic comparison found in  \cite{CutsemYcart}
(see a remark below). 
To that end, with initial  state $z\to\infty$, we focus on the cycles that lie within some range $[\underline{z},z]$, where the truncation parameter
$\underline{z}$ is properly chosen to warrant approximation of the whole process.

The asymptotics (\ref{LT}) implies that there exists a constant $c>0$ such that for all sufficiently large $\underline{z}$
the parameters can be bounded as
\begin{eqnarray*}
\frac{1-c/\underline{z}}{\sqrt{2}}&<&\lambda(z)<\frac{1+c/\underline{z}}{\sqrt{2}}\\
\label{approxparam2}
\frac{1}{\sqrt{2}(1+c/\underline{z})}& <&\theta(z)< 
\frac{1}{\sqrt{2}(1-c/\underline{z})}
\end{eqnarray*}
uniformly in $z>\underline{z}$.
Replacing the variable rate in (\ref{distD})
by constant yields the bounds
$$  \left((1+c/\underline{z})^{-1}\frac{E}{2\sqrt{2}} \right)\wedge (z-\underline{z}) <_{\rm st}  D_z\wedge(z-\underline{z})<_{\rm st} (1-c/\underline{z})^{-1}\frac{E}{2\sqrt{2}},$$
where and henceforth $<_{\rm st}$ denotes the stochastic order. 
Observing that the survival function of $J_z$ is convex, we may bound the jump as 
$$\lambda(z)U<_{\rm st} J_z<_{\rm st}\theta(z)U,$$
whence from (\ref{approxparam2})
 $$ (1+c/\underline{z})^{-1}\frac{U}{\sqrt{2}} <_{\rm st} J_z <_{\rm st} (1-c/\underline{z})^{-1}\frac{U}{\sqrt{2}}\,,~~~~z\geq \underline{z}.$$
From these estimates follow stochastic bounds on the cycle size
\begin{equation}\label{boundsH}
 ((1+c/\underline{z})^{-1}H) \wedge (z-\underline{z})   <_{\rm st} (D_z+J_{z-D_z})\wedge (z-\underline{z}) <_{\rm st} (1-c/\underline{z})^{-1}H\,,~~~~z\geq \underline{z}.
\end{equation}

Setting the bounds (\ref{boundsH}) in terms of multiples of the same random variable $H$ is convenient
in combination with the obvious scaling property: for $d>0$, $R(d\,\cdot)$ is the renewal process with the generic step $dH$.
Let $N(z,\underline{z})$ be the number of cycles of $Z|z$, which fit completely within $[\underline{z},z]$. 
As in { \cite{CutsemYcart}}, from (\ref{boundsH}) we conclude that 
\begin{eqnarray}\label{boundsR} 
 R\left ( (z-\underline{z}) (1-c/\underline{z}))  \right)   <_{\rm st} N(z,\underline{z})<_{\rm st}  R\left ( (z-\underline{z}) (1+c/\underline{z}))\right), ~~~z\geq \underline{z}.
\end{eqnarray}

Letting  $z\to\infty$ then $\underline{z}\to\infty$,  and appealing to $R(z)/z\to \mu^{-1}$ a.s., (\ref{boundsR}) implies 
a weak  law of large numbers for $N(z)$,
\begin{equation}\label{LLN}
\frac{{N(z)}}{z}\dod \frac{1}{\mu}, ~~~z\to\infty.
\end{equation}

We aim next to show the CLT for $N(z)$, that is
\begin{equation}\label{CLTN}
\frac{N(z)-z\mu^{-1}}{\sigma \mu^{-3/2}\sqrt{z}}\dod {\cal N}(0,1), ~~~z\to\infty.
\end{equation}
To that end, we choose $\underline{z}=\omega\sqrt{z}$, where $\omega>0$ is a large parameter. 
Start with splitting
$$N(z)-z\mu^{-1}= (N(z,\underline{z})-(z-\underline{z})\mu^{-1})+(N(z)-N(z,\underline{z})-\underline{z}\mu^{-1}),$$
where $N(z)-N(z,\underline{z})$ counts the cycles that start in $[0,\underline{z}]$;
this component is  annihilated by the scaling, since by (\ref{LLN})
$$\frac{N(z)-N(z,\underline{z})-\underline{z}\mu^{-1}}{\sqrt{\underline{z}}}\dod 0,$$
and the same is true with $\sqrt{\underline{z}}$ replaced by bigger $\sqrt{z}$.
For the leading contribution due to $N(z,\underline{z})$ we obtain using  dominance (\ref{boundsR}) and the CLT for $R(z)$
\begin{eqnarray*}
\prob\left(\frac{N(z,\underline{z})-(z-\underline{z})\mu^{-1}}{\sigma\mu^{-3/2}\sqrt{z}}\leq x\right)\geq 
\prob\left(\frac{R((z-\underline{z})(1+c/\underline{z}))- (z-\underline{z})\mu^{-1}}{\sigma\mu^{-3/2}\sqrt{z}}\leq x\right)=\\
\prob\left(\frac{R((z-\underline{z})(1+c/\underline{z}))- (z-\underline{z})(1+c/\underline{z})\mu^{-1}}{\sigma\mu^{-3/2}\sqrt{z}} + \frac{(z-\underline{z})c}{\omega z \mu^{-1/2} \sigma}       \leq x\right)\to 1-\Phi\left(x-\frac{c}{\omega\sigma\mu^{-1}}\right),
\end{eqnarray*}
as $z\to\infty$. 
Letting $\omega\to\infty$ 
$$\limsup\limits_{z\to\infty}\prob\left( \frac{N(z)-z\mu^{-1}}{\sigma\mu^{-3/2}\sqrt{z}}\leq x\right))=\limsup\limits_{z\to\infty}\prob\left(\frac{N(z,z_0)-z\mu^{-1}}{\sigma\mu^{-3/2}\sqrt{z}}\leq x\right)\geq 1-\Phi(x).$$
The  opposite inequality is derived similarly. 
Hence (\ref{CLTN}) is proved.

\vskip0.2cm
\noindent
{\bf Remark} The renewal-type approximation for  decreasing Markov chains on $\mathbb N$, 
using stochastic comparison appeared in { \cite{CutsemYcart}}. 
However, their Theorem 4.1 on  the  normal limit for the absorption time fails without additional assumptions on the quality of convergence of the step distribution.
For instance, if the decrement in position $z>8$ assumes values  $1$ and $2$ with probabilities $1/2\pm 1/\log z$, the mean absorption time is asymptotic to $2z/3$,
with the remainder being strictly of the order $z/\log z$, 
 therefore not annihilated by the $\sqrt{z}$ scaling.
The error in { \cite{CutsemYcart}} appears on {the} bottom of page 996, where the truncation parameter ($m$, a counterpart of our  $\underline{z}$) is assumed  independent of the initial state.
A recent paper { \cite{AlsmeyerMarynych}},  also concerned with the lattice setting, gives conditions on the rate
of convergence of decrements  in some probability metrics, to ensure  
 the normal approximation of the absorption time. 

\vskip0.2cm
\noindent
{\bf Remark} 
It is of interest to look at the properties of the random set
 $\cal Z$ which, intuitively,  describes an infinite selection process.
This limit object can be interpreted  in the spirit of the boundary theory of  Markov processes: the state space  $[0,\infty)$ has a one-point compactification (the entrance Martin boundary) approached as  the initial state of $Z|z$ tends to $\infty$. 
Applying the coupling argument as in Lemma \ref{cover} one can show that, at large distance from the origin,   $\cal Z$ 
behaves similarly to a stationary  alternating renewal process,  with uniformly distributed gaps and exponential drift intervals. 
The coverage probability and the occupation measure satisfy $p(z)\to 1/2$ and $U([0,z])\sim \sqrt{2}z$, $z\to\infty$.
There have been some related work on Markov processes on the real line which at distance from the origin behave {similarly} to renewal processes  { \cite{Korshunov}}, but  adapting existing results to our problem would require reverting the direction of time.

\vskip0.5cm
\noindent
{\bf  10. Summary} We  summarise  our findings in terms of the original problem. Let 
$L_\varphi (t)$ be the length of increasing subsequence selected by a self-similar strategy with  the acceptance window  of the form (\ref{psi}).

\begin{thm} \label{final}
\noindent
\begin{enumerate} 
\item[\rm(a)] The optimal strategy has the acceptance window of the form {\rm (\ref{psi})}  with

\begin{equation*}
\varphi^*(t)= \sqrt{\frac{2}{t}}-\frac{1}{3t}+O(t^{-3/2}),~~~t\to\infty,
\end{equation*}
and outputs an increasing subsequence with  expected length
$${\mathbb E}L_{\varphi^*}(t)= \sqrt{2t}-\frac{1}{12}\log t+c^*+\frac{\sqrt{2}}{144\sqrt{t}}+O(t^{-1}),~~~t\to\infty,$$
and variance 

$${\rm Var}(L_{\varphi^*}(t))\sim \frac{\sqrt{2t}}{3} +\frac{1}{72}\log t+c_2+ O(t^{-1/2}\log t) .$$

\item[\rm(b)]  The strategy with
$\varphi_0(t):=\sqrt{{2}/{t}}\,\, \wedge  1$
outputs an increasing subsequence with the expected length
$${\mathbb E}L_{\varphi_0} (t)= \sqrt{2t}-\frac{1}{12}\log t+c_0+\frac{\sqrt{2}}{72\sqrt{t}}+ O( t^{-1}),~~~t\to\infty,$$
and variance 
$${\rm Var}(L_{\varphi_0}(t))\sim \frac{\sqrt{2t}}{3} +\frac{1}{24}\log t+c_3+   O(t^{-1/2}\log t) ,~~~t\to\infty.$$

\item [\rm(d)] If  $\varphi(t)\sim \sqrt{2/t} + O(t^{-1})$ then a central limit theorem holds:
$$ \sqrt{3}\,\,\frac{L_\varphi (t)-\sqrt{2t}}{(2t)^{1/4}}\,\dod \,{\cal N}(0,1),~~~t\to\infty.$$

\end{enumerate}
\end{thm}

\vskip0.3cm

\noindent
The instance of part (d) for the  optimal strategy was proved in \cite{BD2}; this can be compared with 
the distributional limit (\ref{limsta}) for the stationary strategy.

Bruss and Delbaen \cite{BD2} used  concavity of $v$ to prove the bounds
$$ \frac{v(t)}{3}\leq {\rm Var}(L_{\varphi^*}(t))\leq \frac{v(t)}{3}+\frac{1}{(\beta-\sqrt{2\beta})\,6\sqrt{2}}\log\frac{t}{\beta}+2, $$
(for $t$ no too small), where $v(\beta)=2$. For large $t$, the logarithmic term in the  lower bound has coefficient  $-1/36$ (as is seen from  (a)) and 
in the upper bound at least $0.55$ (as can be shown by  estimating $\beta$). These bounds can be compared with the coefficient $1/72$ in part (a).
\vskip0.3cm
\noindent
{\bf Remark} The version of the problem with a fixed number of observations $n$ is more complex,
 because  the time of observation $m$ and the running maximum $y$ cannot be aggregated in a single state variable \cite{A2, SS}.
Nevertheless, one can expect that 
the value function is well approximable 
by a function of $(n-m+1)(1-y)$, hence
an analogue of self-similar strategy in  Theorem \ref{final} (b), that is the strategy with acceptance condition
\begin{equation}\label{ac}
0< \frac{x-y}{1-y}<\sqrt{\frac{2}{(n-m+1)(1-y)}}\,\,{\bigwedge}\,\, 1,
\end{equation}
is close to optimality.  Arlotto et al. \cite{A2} employed (\ref{BD1}) and  a de-Poissonisation argument  to show that, indeed, the strategy given by (\ref{ac})
is within $O(\log n)$ from the optimum for $n$ large.
Extending the methods of the present paper,  the latter result has been strengthened recently in \cite{Se}
: the expected length of subsequence selected with acceptance window (\ref{ac})
is  $\sqrt{2n}-(\log n)/12+O(1)$, and this is within $O(1)$ from the optimum.

\vskip0.2cm
\noindent


\end{document}